\documentclass[12pt,a4paper]{article}

\usepackage{latexsym}
\usepackage{amsmath}
\usepackage{amssymb}
\usepackage{arydshln}

\usepackage{cite}
\usepackage{stmaryrd}
\usepackage{enumerate}

\usepackage{hyperref}

\usepackage{color}
\usepackage{lineno}
\usepackage{graphicx}
\usepackage{ae}
\usepackage{amsmath}
\usepackage{amssymb}
\usepackage{latexsym}
\usepackage{url}
\usepackage{epsfig}
\usepackage{mathrsfs}
\usepackage{amsfonts}
\usepackage{amsthm}

\usepackage{float}
\usepackage{subfig}
\usepackage{tikz}
\usetikzlibrary{positioning}

\newtheorem{theorem}{Theorem}[section]

\newtheorem{lemma}[theorem]{Lemma}

\newtheorem{cor}[theorem]{Corollary}

\theoremstyle{definition}

\newtheorem{claim}{Claim}

\newtheorem{problem}{Problem}

\numberwithin{equation}{section} %%  ?   ?    ?
\allowdisplaybreaks    %%  ?  ???

\def\qed{\hfill$\Box$\vspace{12pt}}

\long\def\delete#1{}

\usetikzlibrary{decorations.markings}

\tikzstyle{vertex}=[circle, draw, inner sep=0pt, minimum size=6pt]
\tikzstyle{directed}=[postaction={decorate,
	decoration={markings,mark=at position 0.3 with {\arrow{stealth}}}
}]%%
\usepackage{xcolor}
\usepackage[normalem]{ulem}

\begin{document}
\title {Spectral Tur{\'a}n problem for $t\mathcal{K}_{4}^{-}$-free unbalanced signed graphs}

\author{Linfeng Xie$^{a,b}$,~Xiaogang Liu$^{a,b,}$\thanks{Supported by the National Natural Science Foundation of China (No. 12371358).}~$^,$\thanks{ Corresponding author. Email addresses: xielinfeng@mail.nwpu.edu.cn, xiaogliu@nwpu.edu.cn}~
	\\[2mm]
	{\small $^a$School of Mathematics and Statistics,}\\[-0.8ex]
	{\small Northwestern Polytechnical University, Xi'an, Shaanxi 710072, P.R.~China}\\
	{\small $^b$Xi'an-Budapest Joint Research Center for Combinatorics,}\\[-0.8ex]
	{\small Northwestern Polytechnical University, Xi'an, Shaanxi 710129, P.R. China}
}
\date{}

\openup 0.5\jot
\maketitle

\begin{abstract}
	Let $tK_4$ denote the family of all graphs consisting of $t$ copies of $K_4$ that are allowed to share vertices and $t\mathcal{K}_{4}^{-}$ be the set of all unbalanced signed graphs whose underlying graphs are elements of $tK_4$. In this paper, we characterize the extremal graphs that achieve the maximum index and spectral radius among all $t\mathcal{K}_{4}^{-}$-free unbalanced signed graphs with given order.
	\smallskip
	
	
\end{abstract}

\section{Introduction}
A \emph{signed graph} $\Gamma=(G,\sigma)$ consists of an unsigned graph $G=(V(G),E(G))$ and a sign function $\sigma : E(G)  \to \left \{+1, -1 \right \}$, where $G$ is called the \emph{underlying graph} of $\Gamma$. In 1946, Heider \cite{Hei-JP-1946} first introduced the concept of signed graphs in his study of balance theory in social psychology. Heider sought to explain how friendly/hostile relationships in social networks influence group stability. Based on Heider's work, in 1953, Harary \cite{Har-MichM-1953} first systematically established the mathematical foundation of signed graphs. To date, many results have been obtained on signed graphs. For instance, Chaiken \cite{Chaiken-SIAM-1982} and Zaslavsky \cite{Zaslavsky-DAM-1982} independently derived the matrix-tree theorem for signed graphs. For more information on signed graphs, readers may refer to \cite{Cam-Sei-Tsa-Algebra-1994, Zaslavsky-book-2010, Zaslavsky-EJC-2018}.

Let $\Gamma=(G,\sigma)$ be a signed graph with the vertex set $V(\Gamma)=\left \{v_{1},v_{2},\dots ,v_{n}\right \}$ and the edge set $E(\Gamma)=\left \{e_{1},e_{2},\dots ,e_{m}\right \}$. Denote by $|V(\Gamma)|$ and $|E(\Gamma)|$ the order and the size of $\Gamma$, respectively. Let $N_{\Gamma}(v_i)$ denote the set of neighbors of a vertex $v_i$ in $\Gamma$ and $d_{\Gamma}(v_i)=|N_{\Gamma}(v_i)|$ the degree of $v_i$ in $\Gamma$. An edge $v_i v_j$ is called a \emph{positive edge} (respectively, \emph{negative edge}) if $\sigma(v_i v_j)=+1$ (respectively, $\sigma(v_i v_j)=-1$). A \emph{subgraph} of $\Gamma$ is a subgraph that preserves the sign of every edge of $\Gamma$. The \emph{adjacency matrix} of $\Gamma$ is defined as $A(\Gamma)=(a_{ij}^{\sigma})$, where $a_{ij}^{\sigma}=\sigma(v_{i}v_{j})a_{ij}$ and $a_{ij}=1$ if $ v_{i} $ and $v_{j}$ are adjacent, and $a_{ij}=0$ otherwise. The eigenvalues of $A(\Gamma)$ are denoted by
$$
\lambda_{1}(\Gamma)\ge \lambda_{2}(\Gamma)\ge \dots \ge \lambda_{n}(\Gamma),
$$
which are called the \emph{adjacency spectrum} of $\Gamma$. The \emph{index} of $\Gamma$ is $\lambda_{1}(\Gamma)$. The \emph{spectral radius} of $\Gamma$ is defined as
$$
\rho(\Gamma)=\max\left \{\lambda_{1}(\Gamma),-\lambda_{n}(\Gamma)\right \}.
$$

Let $U$ be a nonempty proper subset of $V(\Gamma)$. The operation of reversing the signs of all edges between $U$ and $V(\Gamma) \setminus U$ is called a \emph{switching} of $\Gamma$. A signed graph $\Gamma^{\prime}$ is said to be \emph{switching equivalent} to $\Gamma$ if $ \Gamma^{\prime}$ can be obtained from $\Gamma$ by a finite sequence of switchings, denoted by $\Gamma\sim \Gamma^{\prime}$. A cycle is called \emph{positive} if the number of its negative edges is even; otherwise, \emph{negative}. A signed graph is called \emph{balanced} if each of its cycles is positive; otherwise, \emph{unbalanced}.

For a given family of signed graphs $\mathcal{F}$, a signed graph $\Gamma$ is called \emph{$\mathcal{F}$-free} if $\Gamma$ contains no subgraph isomorphic to any member of $\mathcal{F}$. The problem of determining the maximum index and spectral radius of an $\mathcal{F}$-free signed graph of order $n$ is of great interest, and the corresponding unsigned version was first termed the spectral Tur{\'a}n problem by Nikiforov \cite{Nikiforov-LAA-2010} in 2010. For more information on the unsigned version, we refer the readers to \cite{Nikiforov-LAA-2007, Ni-Wang-Kang-EJC-2023, Yuan-Wang-Zhai-EJLA-2012, Chen-Lei-Li-EJC-2025, Nikiforov-LAA-2017}.

Let $\mathcal{C}_{r}^-$ and $\mathcal{K}_{r}^-$ be the sets of all unbalanced signed graphs with underlying graphs $C_r$ and $K_r$, respectively. Recently, in 2024, Wang and Hou \cite{Wang-Hou-Li-LAA-2024} determined the maximum spectral radius of $\mathcal{C}_{3}^-$-free graphs. In the same year, Wang and Lin \cite{Wang-Lin-DAM-2024} gave the maximum index of $\mathcal{C}_{4}^-$-free graphs. In 2025, Wang, Hou and Huang  \cite{Wang-Hou-Huang-DAM-2025} determined the maximum index of $\mathcal{C}_{2k+1}^-$-free graphs for $3\le k\le \frac{n}{10}-1$. In 2024, Chen and Yuan \cite{Chen-Yuan-AMC-2024} determined the maximum spectral radius of $\mathcal{K}_{4}^-$-free graphs, and Wang \cite{Wang-LAA-2024} gave the maximum spectral radius of $\mathcal{K}_{5}^-$-free graphs. In the same year, Xiong and Hou \cite{Xiong-Hou-AC-2024} determined the maximum index of $\mathcal{K}_{r+1}^-$-free graphs for $3\le r\le n-1$.

Let $tK_{s+1}$ denote the family of all graphs consisting of $t$ copies of $K_{s+1}$ that are allowed to share vertices and $t\mathcal{K}_{s+1}^{-}$ be the set of all unbalanced signed graphs whose underlying graphs are elements of $tK_{s+1}$. Recently, Li and Qin \cite{Li-Qin-Ar-2025} determined the maximum index of $t\mathcal{K}_{3}^-$-free graphs. Motivated by this result, it is natural to consider the following problem.

\begin{problem}\label{Problem1-1-1}
	For  integers $t\ge 2$ and $s\ge 3$, what is the maximum index among all $t\mathcal{K}_{s+1}^{-}$-free unbalanced signed graphs?
\end{problem}

We would like to point out that if $t$ vertex-disjoint copies of $K_{s+1}$ appear in Problem \ref{Problem1-1-1}, by Theorem \ref{maximum index}, the signed complete graph with exactly one negative edge achieves the maximum index.

\begin{theorem}\emph{(See \cite[Theorem~4.7]{Brunetti-Stanic-CAM-2022})}\label{maximum index}
	Let $\Gamma$ be an unbalanced signed graph of order $n$, and let $\Gamma^{\prime}$ be a signed complete graph of order $n$ with exactly one negative edge. Then $\lambda_{1}(\Gamma)\le \lambda_{1}(\Gamma^{\prime})$ with equality holding if and only if $\Gamma \sim \Gamma^{\prime}$.
\end{theorem}

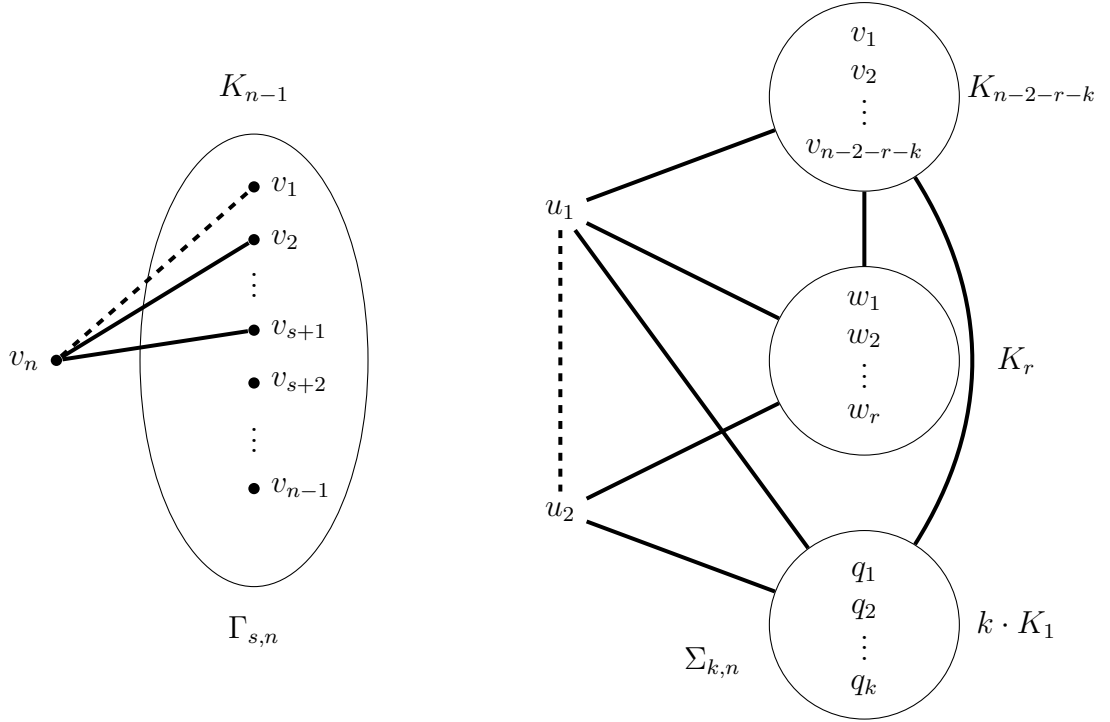
\begin{figure}[H]
	\hspace{-1.5cm} % 仅整体左移，不改变图形任何大小
	\begin{minipage}[c]{0.6\textwidth}
		\hspace{2.5cm}
		\begin{tikzpicture}[scale=1.0]
			
			% 椭圆（稍微放大）
			\draw (2,0) ellipse (1.5 and 3);
			
			% 标题
			\node at (2,3.6) {$K_{n-1}$};
			\node at (2,-3.6) {$\Gamma_{s,n}$};
			
			% 左侧顶点 u
			\node[circle,fill,inner sep=1.5pt,label=left:$v_n$] (v_n) at (-0.6,0) {};
			
			% 椭圆内顶点（整体稍微向中间收）
			\node[circle,fill,inner sep=1.5pt,label=right:$v_1$] (v1) at (2,2.3) {};
			\node[circle,fill,inner sep=1.5pt,label=right:$v_2$] (v2) at (2,1.6) {};
			\node at (2,1.1) {$\vdots$};
			\node[circle,fill,inner sep=1.5pt,label=right:$v_{s+1}$] (vtm) at (2,0.4) {};
			\node[circle,fill,inner sep=1.5pt,label=right:$v_{s+2}$] (vt) at (2,-0.3) {};
			\node at (2,-0.95) {$\vdots$};
			\node[circle,fill,inner sep=1.5pt,label=right:$v_{n-1}$] (v_{n-1}) at (2,-1.7) {}; % 上移
			
			% 连边
			\draw[line width=1.5pt][dashed] (v_n) -- (v1);
			\draw [line width=1.5pt](v_n) -- (v2);
			\draw [line width=1.5pt](v_n) -- (vtm);
		\end{tikzpicture}
	\end{minipage}
	\hfill
	\begin{minipage}[c]{1\textwidth}
		\begin{tikzpicture}[scale =1]
			% 左边两个顶点：u1、u2（距离已拉近）
			\node (u1) at (0,2)  {$u_{1}$};
			\node (u2) at (0,-2) {$u_{2}$};
			
			% 右边上方大圆：v1, v2, ..., vt
			\node at (6.2,3.6) {$K_{n-2-r-k}$};
			\node[circle, draw, minimum size=2.5cm] (V) at (4,3.5) {};
			\node at (4,4.3) {$v_1$};
			\node at (4,3.8) {$v_2$};
			\node at (4,3.4) {$\vdots$};
			\node at (4,2.8) {$v_{n-2-r-k}$};
			
			% 右边中间大圆：w1, w2, ..., wr
			\node at (6,0) {$K_{r}$};
			\node[circle, draw, minimum size=2.5cm] (W) at (4,0) {};
			\node at (4,0.8) {$w_1$};
			\node at (4,0.3) {$w_2$};
			\node at (4,-0.1) {$\vdots$};
			\node at (4,-0.7) {$w_r$};
			
			% 右边下方大圆：q1, q2, ..., qk
			\node at (6,-3.5) {$k\cdot K_1$};
			\node[circle, draw, minimum size=2.5cm] (Q) at (4,-3.5) {};
			\node at (4,-2.8) {$q_1$};
			\node at (4,-3.3) {$q_2$};
			\node at (4,-3.7) {$\vdots$};
			\node at (4,-4.3) {$q_k$};
			
			% 连线（和你原来代码格式一致）
			\draw [line width=1.5pt][dashed] (u1) -- (u2);
			\draw [line width=1.5pt]  (u1) -- (V);
			\draw [line width=1.5pt]  (u1) -- (W);
			\draw [line width=1.5pt] (u1) -- (Q);
			\draw [line width=1.5pt]  (u2) -- (W);
			\draw [line width=1.5pt]  (u2) -- (Q);
			\draw [line width=1.5pt]  (V) -- (W);
			\draw [line width=1.5pt, bend left=32] (V) to (Q);
			
			\node at (2,-4) {$\Sigma_{k,n}$};
		\end{tikzpicture}
		\hspace{2cm}
	\end{minipage}
	\caption{The signed graphs $\Gamma_{s,n}$ and $\Sigma_{k,n}$.}
	\label{figure1} 	
\end{figure}

As shown in Figure \ref{figure1}, let $K_{n-1}$ be a complete graph with vertex set $\left\{v_1,v_2,\dots,v_{n-1}\right\}$, and let $\Gamma_{s,n}$ be the signed graph obtained by adding a new vertex $v_n$ to $K_{n-1}$ and joining $v_n$ to the $s+1$ vertices $v_1,v_2,\dots,v_{s+1}$, where $v_n v_1$ is the unique negative edge.

In this paper, we investigate the case $s=3$ for Problem \ref{Problem1-1-1}. We state our main result as follows.

\begin{theorem}\label{tK4}
	Let $t\ge 2$ be a positive integer such that $r=\frac{1+\sqrt{8t-7}}{2}$ is an integer. If $\Gamma$ is a $t\mathcal{K}_{4}^{-}$-free unbalanced signed graph of sufficiently large order $n$, then
	\[
	\lambda_1(\Gamma )\le
	\begin{cases}
		\lambda_1(\Gamma_{r,n}), & \text{if}~~ 2 \le t \le \begin{pmatrix}
			n-2\\2
		\end{pmatrix}, \\[6pt]
		\lambda_1(\Gamma_{n-2,n}),   & \text{if}~~ t \ge\begin{pmatrix}
			n-2\\2
		\end{pmatrix}+1,
	\end{cases}
	\]
	with equality holding if and only if
	\[
	\Gamma \sim
	\begin{cases}
		\Gamma_{r,n}, & \text{if }~~ 2 \le t \le \begin{pmatrix}
			n-2\\2
		\end{pmatrix}, \\[6pt]
		\Gamma_{n-2,n},   & \text{if}~~ t \ge \begin{pmatrix}
			n-2\\2
		\end{pmatrix}+1.
	\end{cases}
	\]
\end{theorem}

The paper is organized as follows. In Section $2$, we present some lemmas to be used in the proof of Theorem \ref{tK4}. In Section $3$, we give the proof of Theorem \ref{tK4} and characterize the extremal graphs that achieve the maximum spectral radius among all $t\mathcal{K}_{4}^{-}$-free unbalanced signed graphs with given order.

\section{Preliminaries}

\begin{lemma}\emph{(See \cite[Proposition~1.11]{Xiong-Hou-AC-2024})}\label{Xiong-Hou-AC-2024-2}
	Let $\Gamma_{s,n}$ be the signed graph of order $n$ as shown in Figure \ref{figure1}. Then $\lambda_1(\Gamma_{s,n})$ is the largest root of the polynomial
	$$
	f_{s,n}(\lambda)=\lambda^{3}-(n-3)\lambda^{2}-(n+s-1)\lambda-s^2+n+ns-3
	$$
	and satisfies
	$$
	n-2\le \lambda_1(\Gamma_{s,n})<n-1,
	$$
	with equality on the left if and only if $s=1$.
\end{lemma}

By Lemma \ref{Xiong-Hou-AC-2024-2}, we obtain the following corollary.

\begin{cor}\label{Xiong-Hou-AC-2024-3}
	$\lambda_1(\Gamma_{s,n})<\lambda_1(\Gamma_{s+1,n})$ for $s\ge 1$.
\end{cor}
\begin{proof}
	By Lemma \ref{Xiong-Hou-AC-2024-2}, we have
	\[
	f_{s+1,n}(\lambda)=\lambda^{3}-(n-3)\lambda^{2}-(n+s)\lambda-s^2-2s+2n+ns-4.	
	\]
	Then
	\[
	f_{s+1,n}(\lambda)-f_{s,n}(\lambda)=-\lambda+n-2s-1.
	\]
	Since $s\ge 1$ and $\lambda_1(\Gamma_{s,n})\ge n-2$, we have
	\begin{align*} f_{s+1,n}(\lambda_1(\Gamma_{s,n}))-f_{s,n}(\lambda_1(\Gamma_{s,n}))&=-\lambda_1(\Gamma_{s,n})+n-2s-1\\
		&\le -2s+1\\
		&<0.
	\end{align*}
	Hence, $\lambda_1(\Gamma_{s+1,n})>\lambda_1(\Gamma_{s,n})$. \qed
\end{proof}

\begin{lemma}\emph{(See \cite[Theorem~1.10]{Xiong-Hou-AC-2024})}\label{Xiong-Hou-AC-2024-1}
	If $\Gamma$ is a $\mathcal{K}_{s+1}^{-}$-free $(3\le s\le n-1)$ unbalanced signed graph of order $n$, then
	$$
	\lambda_1(\Gamma)\le \lambda_1(\Gamma_{s-2,n}),
	$$
	with equality holding if and only if $\Gamma \sim \Gamma_{s-2,n}$.
\end{lemma}

\begin{lemma}\emph{(See \cite[Theorem~1.3]{Wang-Hou-Li-LAA-2024})}\label{Wang-Hou-Li-LAA-2024}
	Let $\Gamma$ be a connected unbalanced signed graph of order $n$. If $\Gamma$ is $C_3^-$-free, then
	$$
	\rho(\Gamma)\le \frac{1}{2}\left(\sqrt{n^{2}-8}+n-4\right).
	$$
\end{lemma}

\begin{lemma}\emph{(See \cite[Lemma~2.5]{Sun-Liu-Lan-LAA-2022})}\label{Sun-Liu-Lan-LAA-2022}
	Let $\Gamma$ be a signed graph of order $n$. Then there exists a signed graph $\Gamma^{\prime}$ switching equivalent to $\Gamma$ for which the largest eigenvalue $\lambda_{1}(\Gamma^{\prime})$ admits a non-negative eigenvector.
\end{lemma}

\begin{lemma}\emph{(See \cite[Proposition~3.2]{Zaslavsky-DAM-1982})}\label{Zaslavsky-DAM-1982}
	Two signed graphs on the same underlying graph are switching equivalent if and only if they have the same list of balanced cycles.
\end{lemma}

\begin{lemma}\emph{(See \cite[Lemma~2.1]{Hou-Tang-Wang-AMC-2019})}
	\label{Hou-Tang-Wang-AMC-2019}
	Let $\Gamma_{1}=(G,\sigma_{1})$ and $\Gamma_{2}=(G,\sigma_{2})$ be two signed graphs on the same underlying graph $G$. The following are equivalent:
	
	$\mathrm{(1)}$~$\Gamma_{1}$ and $\Gamma_{2}$ are switching equivalent.
	
	$\mathrm{(2)}$~$A(\Gamma_{1})$ and $A(\Gamma_{2})$ are similar.
\end{lemma}

Let $M$ be a real symmetric matrix of order $n$, and let $[n] = \{1, 2, \dots, n\}$. Given a partition $\Pi: [n] = X_1 \cup X_2 \cup \dots \cup X_k$, the matrix $M$ can be partitioned  as
\[
M =
\begin{pmatrix}
	M_{1,1} & M_{1,2} & \cdots & M_{1,k} \\
	M_{2,1} & M_{2,2} & \cdots & M_{2,k} \\
	\vdots & \vdots & \ddots & \vdots \\
	M_{k,1} & M_{k,2} & \cdots & M_{k,k}
\end{pmatrix},
\]
where $M_{ij}$ is the submatrix of $M$ with rows indexed by $X_i$ and columns indexed by $X_j$. The \textit{characteristic matrix} $\chi_{\Pi}$ of $\Pi$ is the \( n \times k\) matrix whose columns are the characteristic vectors of $X_1, \dots, X_k$. If all row sums of $M_{i,j}$ are equal to the same constant, denoted by $b_{i,j}$, for all $i,j \in \{1, 2, \dots, k\}$, then $\Pi$ is called an \emph{equitable partition} of $M$, and the matrix $Q= (b_{i,j})_{i,j=1}^k$ is called an \emph{equitable quotient matrix} of $M$.

\begin{lemma} \emph{(See \cite[p. 30]{Brouwer-Haemers-Book-2011})}\label{Brouwer-Haemers-Book-2011}
	The matrix $M$ has the following two kinds of eigenvectors and eigenvalues:
	
	$\mathrm{(i)}$~Eigenvectors lying in the column space of $\chi_{\Pi}$; their corresponding eigenvalues coincide with those of $Q$.
	
	$\mathrm{(ii)}$~Eigenvectors orthogonal to the columns of $\chi_{\Pi}$; the corresponding eigenvalues of $M$ remain unchanged if some scalar multiple of the all-one
	block is added to each block $M_{i,j}$ for all $i,j \in \{1,\ldots,m\}$.
\end{lemma}

See Figure \ref{figure1}. Let $K_{r}$ be a complete graph with vertex set $\left\{w_1,w_2,\dots,w_r\right\}$, $K_{n-2-r-k}$ a complete graph with vertex set $\left\{v_1,v_2,\dots,v_{n-2-r-k}\right\}$, and $k\cdot K_1$ an empty graph on $k$ vertices. Let $\Sigma_{k,n}$ be the signed graph obtained from $K_{r}$, $K_{n-2-r-k}$ and $k\cdot K_1$ by adding two new vertices $u_1$, $u_2$, joining $u_1$ to $u_2$, joining $u_1$ to all vertices of $K_{r}$, $K_{n-2-r-k}$ and $k\cdot K_1$, joining $u_2$ to $K_{n-2-r-k}$ and $k\cdot K_1$, and finally adding all edges between $K_{r}$ and $K_{n-2-r-k}$ and between $k\cdot K_1$ and $K_{n-2-r-k}$, where $u_1 u_2$ is the unique negative edge.

\begin{lemma}\label{sigma_{k,n}}
	Let $\Sigma_{k,n}$ be the signed graph of sufficiently large order $n$ as shown in Figure \ref{figure1}. If $k\ge 2$, $r\ge 2$ and $\lambda_1(\Sigma_{k,n})>n-2$, then
	\[
	\lambda_1(\Sigma _{k,n})<\lambda_1(\Sigma_{k-1,n}).
	\]
\end{lemma}
\begin{proof}
	By using the vertex partition $V_1=\left \{u_2\right \}$, $V_2=\left \{u_1\right \}$, $V_3=\left \{w_1,w_2,\dots,w_r\right \}$, $V_4=\left \{q_1,q_2,\dots,q_k\right \}$ and $V_5=\left \{v_1,v_2,\dots,v_{n-2-r-k}\right \}$, we present $A(\Sigma_{k,n})$ and its corresponding quotient matrix $Q(\Sigma_{k,n})$ as follows
	\[
	A(\Sigma_{k,n})=
	\begin{pmatrix}
		0 & -1 & \mathbf{j}^T & \mathbf{j}^T & \mathbf{0} \\
		-1 & 0 & \mathbf{j}^T & \mathbf{j}^T & \mathbf{j}^T \\
		\mathbf{j} & \mathbf{j} & J-I & \mathbf{0} & J \\
		\mathbf{j} & \mathbf{j} & \mathbf{0} & \mathbf{0} & J \\
		\mathbf{0} & \mathbf{j} & J & J	 & J-I
	\end{pmatrix}
	\]
	and
	\[
	Q(\Sigma_{k,n})=
	\begin{pmatrix}
		0 & -1 & r & k & 0 \\
		-1 & 0 & r & k & n-2-r-k \\
		1 & 1 & r-1 & 0 & n-2-r-k \\
		1 & 1 & 0 & 0 & n-2-r-k \\
		0 & 1 & r & k & n-3-r-k
	\end{pmatrix},
	\]
	where $\mathbf{0}$ and $\mathbf{j}$ represent the zero vector and the all-ones vector of appropriate dimensions, respectively, and $I$ and $J$ denote the identity matrix and all-ones matrix of appropriate orders, respectively.
	We add some scalar multiple of $\mathbf{j}$ or $J$ to the blocks equal to $-1$, $\mathbf{j}$, $J$ and $J-I$ in $A(\Sigma_{k,n})$. Then we have
	\[
	A^{\prime}(\Sigma_{k,n})=
	\begin{pmatrix}
		0 & 0 & \mathbf{0} & \mathbf{0} & \mathbf{0} \\
		0 & 0 & \mathbf{0} & \mathbf{0} & \mathbf{0} \\
		\mathbf{0} & \mathbf{0} & -I & \mathbf{0} & \mathbf{0} \\
		\mathbf{0} & \mathbf{0} & \mathbf{0} & \mathbf{0} & \mathbf{0} \\
		\mathbf{0} & \mathbf{0} & \mathbf{0} & \mathbf{0}	 & -I
	\end{pmatrix}.
	\]
	The eigenvalues of $A^{\prime}(\Sigma_{k,n})$ are $-1$ (with multiplicity $n-k-2$) and $0$ (with multiplicity $k+2$). By Lemma \ref{Brouwer-Haemers-Book-2011} and $\lambda_1(\Sigma_{k,n})>0$, we obtain that $\lambda_1(\Sigma _{k,n})=\lambda_1(Q(\Sigma_{k,n}))$, that is, the largest eigenvalue of $Q(\Sigma_{k,n})$.
	By direct calculation, the characteristic polynomial of $Q(\Sigma_{k,n})$ is
	\[
	\begin{aligned}
		h_{k,n}(x) &= x^5 + (k - n + 4)x^4 \\
		&\quad + (k^2 - kn + kr + 2k - 2n - r + 4)x^3 \\
		&\quad + (-k^2r + k^2 + knr - kn - kr^2 + nr - r^2 - r - 2)x^2 \\
		&\quad + (-2k^2 + 2kn - 3kr - 3k + nr + n - r^2 - 3)x \\
		&\quad + (2k^2r - 2k^2 - 2knr + 2kn + 2kr^2 - 2k).
	\end{aligned}
	\]
	Hence, we have
	\[
	\begin{aligned}
		h_{k-1,n}(x)-h_{k,n}(x) &= -x^4 + (-2k + n - r - 1)x^3 \\
		&\quad + (r^2 + 2kr - 2k - nr + n - r + 1)x^2 \\
		&\quad + (4k - 2n + 3r + 1)x \\
		&\quad - 2r^2 - 4kr + 4k + 2nr - 2n + 2r.
	\end{aligned}
	\]
	Set
	$$
	m(x)=-x^4+ (-2k + n - r - 1)x^3+ (r^2 + 2kr - 2k - nr + n - r + 1)x^2.
	$$
	Since $k\ge 2$, $r\ge 2$, $n\ge k+r+2$ and $\lambda_1(\Sigma_{k,n})>n-2$, we obtain that
	\begin{align*}
		-\lambda_1(\Sigma _{k,n})^4+(-2k + n - r - 1)\lambda_1(\Sigma _{k,n})^3&=\lambda_1(\Sigma _{k,n})^3(-\lambda_1(\Sigma _{k,n})-2k+n-r-1)\\
		&<-(2k+r-1)\lambda_1(\Sigma _{k,n})^3,
	\end{align*}
	and
	\begin{align*}
		(r^2 + 2kr - 2k - nr + n - r + 1)\lambda_1(\Sigma _{k,n})^2
		&=\left((r-1)(2k-n)+r^2-r+1\right)\lambda_1(\Sigma _{k,n})^2\\
		&\le \left((r-1)(k-r-2)+r^2-r+1\right)\lambda_1(\Sigma _{k,n})^2\\
		&=(rk-2r-k+3)\lambda_1(\Sigma _{k,n})^2.
	\end{align*}
	Hence,
	\begin{align*}
		m(\lambda_1(\Sigma _{k,n}))&<-(2k+r-1)\lambda_1(\Sigma _{k,n})^3+\left( rk-2r-k+3 \right)\lambda_1(\Sigma _{k,n})^2\\
		&< \lambda_1(\Sigma _{k,n})^2\left(-(n-2)(2k+r-1)+(rk-2r-k+3)\right)\\
		&<\left(k(-2n+r+3)-nr+n+1\right)\lambda_1(\Sigma _{k,n})^2\\
		&\le \left(k(-n-k+1)-nr+n+1\right)\lambda_1(\Sigma _{k,n})^2\\
		&\le (-3n-1)\lambda_1(\Sigma _{k,n})^2.
	\end{align*}
	Since the linear term and constant term of $h_{k-1,n}(x)-h_{k,n}(x)$ are at most quadratic in $n$, we have $h_{k-1,n}(\lambda_1(\Sigma _{k,n}))-h_{k,n}(\lambda_1(\Sigma _{k,n}))<0$ for sufficiently large $n$. Hence, $\lambda_1(\Sigma_{k-1,n})> \lambda_1(\Sigma _{k,n})$.\qed
\end{proof}

\section{Proof~of~Theorem~\ref{tK4}}

\begin{Tproof}\textbf{~of~Theorem~\ref{tK4}.}~When $t\ge \begin{pmatrix}
		n-2\\2
	\end{pmatrix}+1$, it is straightforward to verify that $\Gamma_{n-2,n}$ is unbalanced and $t\mathcal{K}_{4}^{-}$-free. By Theorem \ref{maximum index}, $\lambda_1(\Gamma)\le \lambda_1(\Gamma_{n-2,n})$ with equality holding if and only if  $\Gamma \sim \Gamma_{n-2,n}$.
	
	Next, we consider the case where $2\le t \le \begin{pmatrix}
		n-2\\2
	\end{pmatrix}$. Let $\Gamma=(G,\sigma)$ be a signed graph having the maximum index over all $t\mathcal{K}_{4}^{-}$-free unbalanced signed graph of order $n$. By Lemma \ref{Sun-Liu-Lan-LAA-2022}, let $\Gamma^{\prime}=(G,\sigma^{\prime})$ be the signed graph switching equivalent to $\Gamma$ for which $\lambda_{1}(\Gamma^{\prime})$ admits a non-negative eigenvector. Let $V(\Gamma^{\prime}) = \left \{v_1,v_2,\dots,v_n \right \}$ and $\mathbf{x}=\left (x_1,x_2,\dots,x_n \right )^T$ be the non-negative unit eigenvector of $A(\Gamma^{\prime})$ corresponding to $\lambda_1(\Gamma^{\prime})$. By Lemmas \ref{Zaslavsky-DAM-1982} and \ref{Hou-Tang-Wang-AMC-2019}, $\Gamma^{\prime}=(G,\sigma^{\prime})$ is also $t\mathcal{K}_{4}^{-}$-free unbalanced signed graph with the maximum index. Note that $\Gamma_{2,n}$ is $t\mathcal{K}_{4}^{-}$-free unbalanced signed graph. By Lemma \ref{Xiong-Hou-AC-2024-2},
	$$
	\lambda_1(\Gamma^{\prime})\ge \lambda_1(\Gamma_{2,n})> \lambda_1(\Gamma_{1,n})=n-2 > \frac{1}{2}\left(\sqrt{n^{2}-8}+n-4\right).
	$$
	By Lemma \ref{Wang-Hou-Li-LAA-2024}, $\Gamma^{\prime}$ contains a subgraph $C_3^-$. Assume that $V(C_3^-)=\left \{v_1,v_2,v_n  \right \}$.
	\begin{claim}\label{one zero coordinate}
		The eigenvector $\mathbf{x}$ has at most one zero coordinate.
	\end{claim}
	\noindent\emph{Proof of Claim 1.}~If $\mathbf{x}$ has two zero coordinates, say $x_i$ and $x_j$. Then
	$$
	\lambda_1(\Gamma^{\prime})=\mathbf{x}^{T}A(\Gamma^{\prime})\mathbf{x}\le \lambda_1(\Gamma^{\prime}-v_i-v_j)\le \lambda_1(K_{n-2},+)=n-3,
	$$
	where $(K_{n-2}, +)$ denotes the complete signed graph of order $n-2$ with all positive edges, a contradiction.
	
	\begin{claim}\label{connected}
		$\Gamma^{\prime}$ is connected.
	\end{claim}
	\noindent\emph{Proof of Claim 2.}~By contradiction, assume that $\Gamma^{\prime}_1$ and $\Gamma^{\prime}_2$ are two distinct connected components of $\Gamma^{\prime}$, where $\lambda_1(\Gamma^{\prime})=\lambda_1(\Gamma^{\prime}_1)$. Without loss of generality, suppose that $v_i \in V(\Gamma^{\prime}_1)$ and $v_j \in V(\Gamma^{\prime}_2)$. Let $\Gamma^{*}$ be the graph obtained from $\Gamma^{\prime}$ by adding a positive edge $v_i v_j$. Clearly, $\Gamma^{*}$ is also unbalanced and $t\mathcal{K}_{4}^{-}$-free. By the Rayleigh principle, we have
	$$
	\lambda_1(\Gamma^{*})-\lambda_1(\Gamma^{\prime})\ge \mathbf{x}^{T}A(\Gamma^{*})\mathbf{x}-\mathbf{x}^{T}A(\Gamma^{\prime})\mathbf{x}=2x_i x_j\ge 0.
	$$
	If $\lambda_1(\Gamma^{*})=\lambda_1(\Gamma^{\prime})$, then $\mathbf{x}$ is also an eigenvector of $A(\Gamma^{*})$ corresponding to $\lambda_1(\Gamma^{*})$. Note that
	$$
	\lambda_1(\Gamma^{\prime})x_i=\sum_{v_k \in N_{\Gamma^{\prime}}(v_i)}\sigma^{\prime}(v_k v_i)x_k, ~~
	\lambda_1(\Gamma^{\prime})x_j=\sum_{v_k \in N_{\Gamma^{\prime}}(v_j)}\sigma^{\prime}(v_k v_j)x_k
	$$
	and
	$$
	\lambda_1(\Gamma^{*})x_i=\sum_{v_k \in N_{\Gamma^{\prime}}(v_i)}\sigma^{\prime}(v_k v_i)x_k+x_j, ~~
	\lambda_1(\Gamma^{*})x_j=\sum_{v_k \in N_{\Gamma^{\prime}}(v_j)}\sigma^{\prime}(v_k v_j)x_k+x_i.
	$$
	Then we have $x_i=x_j=0$, which contradicts  Claim \ref{one zero coordinate}. Hence, $\lambda_1(\Gamma^{*})>\lambda_1(\Gamma^{\prime})$, which contradicts the maximality of $\lambda_1(\Gamma^{\prime})$.
	
	\textbf{Notation.}~For $U\subseteq V(\Gamma)$, let $\Gamma[U]$ denote the signed subgraph of $\Gamma$ induced by $U$, with edge signs inherited from $\Gamma$. Sometimes, we say that $U$ induces $\Gamma[U]$.
	
	\begin{claim}
		The $C_3^-$ contains all negative edges of $\Gamma^{\prime}$.
	\end{claim}
	\noindent\emph{Proof of Claim 3.}~If there is a negative edge $v_i v_j \in E(\Gamma[V(\Gamma^{\prime})\setminus V(C_3^-)])$, then construct a new graph $\Gamma^{*}$ obtained from $\Gamma^{\prime}$ by deleting $v_i v_j$. Obviously, $\Gamma^{*}$ is also unbalanced and $t\mathcal{K}_{4}^{-}$-free. By the Rayleigh principle,
	$$
	\lambda_1(\Gamma^{*})-\lambda_1(\Gamma^{\prime})\ge \mathbf{x}^{T}A(\Gamma^{*})\mathbf{x}-\mathbf{x}^{T}A(\Gamma^{\prime})\mathbf{x}=2x_i x_j\ge 0.
	$$
	If $\lambda_1(\Gamma^{*})=\lambda_1(\Gamma^{\prime})$, then $\mathbf{x}$ is also an eigenvector of $A(\Gamma^{*})$ corresponding to $\lambda_1(\Gamma^{*})$. Note that
	$$
	\lambda_1(\Gamma^{\prime})x_i=\sum_{v_k \in N_{\Gamma^{\prime}}(v_i)}\sigma^{\prime}(v_k v_i)x_k, ~~
	\lambda_1(\Gamma^{\prime})x_j=\sum_{v_k \in N_{\Gamma^{\prime}}(v_j)}\sigma^{\prime}(v_k v_j)x_k
	$$
	and
	$$
	\lambda_1(\Gamma^{*})x_i=\sum_{v_k \in N_{\Gamma^{\prime}}(v_i)}\sigma^{\prime}(v_k v_i)x_k+x_j, ~~
	\lambda_1(\Gamma^{*})x_j=\sum_{v_k \in N_{\Gamma^{\prime}}(v_j)}\sigma^{\prime}(v_k v_j)x_k+x_i.
	$$
	Then we have $x_i=x_j=0$, which contradicts Claim \ref{one zero coordinate}. Hence, $\lambda_1(\Gamma^{*})>\lambda_1(\Gamma^{\prime})$, which contradicts the maximality of $\lambda_1(\Gamma^{\prime})$.
	
	\begin{claim}\label{one negative edge}
		$\Gamma^{\prime}$ contains exactly one negative edge.
	\end{claim}
	\noindent\emph{Proof of Claim 4.}~By contradiction, assume that $C_3^-$ contains three negative edges of $\Gamma^{\prime}$. Let $\Gamma^{*}$ be the graph obtained from $\Gamma^{\prime}$ by reversing the signs of the negative edges $v_1 v_n$ and $v_2 v_n$.
	Clearly, for every $K_4^-$ in $\Gamma^{*}$,  its vertices also form a $K_4^-$ in $\Gamma^{\prime}$. Hence, $\Gamma^{*}$ is also unbalanced and $t\mathcal{K}_{4}^{-}$-free. By the Rayleigh principle,
	\[
	\lambda_1(\Gamma^{*})-\lambda_1(\Gamma^{\prime})\ge \mathbf{x}^{T}A(\Gamma^{*})\mathbf{x}-\mathbf{x}^{T}A(\Gamma^{\prime})\mathbf{x}=4x_n(x_1+x_2)\ge 0.
	\]
	If $\lambda_1(\Gamma^{*})=\lambda_1(\Gamma^{\prime})$, then $\mathbf{x}$ is also an eigenvector of $A(\Gamma^{*})$ corresponding to $\lambda_1(\Gamma^{*})$ and $x_n=0$. Note that
	$$
	\lambda_1(\Gamma^{\prime})x_n=\sum_{v_k \in N_{\Gamma^{\prime}}(v_n)}\sigma^{\prime}(v_k v_n)x_k
	$$
	and
	$$
	\lambda_1(\Gamma^{*})x_n=\sum_{v_k \in  N_{\Gamma^{\prime}}(v_n)}\sigma^{\prime}(v_k v_n)x_k+2x_1+2x_2.
	$$
	Then we have $x_1=x_2=x_n=0$, which contradicts Claim \ref{one zero coordinate}. Hence, $\lambda_1(\Gamma^{*})>\lambda_1(\Gamma^{\prime})$, which contradicts the maximality of $\lambda_1(\Gamma^{\prime})$.
	
	Without loss of generality, we suppose in the sequel that $v_1 v_n$ is the unique negative edge in $\Gamma^{\prime}$ and that $x_n\le x_1$. By Claim \ref{one zero coordinate}, $x_1>0$.
	\begin{claim}\label{x>=0}
		$x_i >0$ for $2\le i \le n-1$.
	\end{claim}
	\noindent\emph{Proof of Claim 5.}~By contradiction, assume that $x_i=0$ for some $i$. By Claims \ref{connected} and \ref{one negative edge}, we have $d_{\Gamma^{\prime}}(v_i)\ge 1$ and all edges incident to $v_i$ are positive. Hence,
	$$
	0=\lambda_1(\Gamma^{\prime})x_i=\sum_{v_k \in N_{\Gamma^{\prime}}(v_i)}\sigma^{\prime}(v_k v_i)x_k>0,
	$$
	a contradiction.
	\begin{claim}
		$N_{\Gamma^{\prime}}(v_n)\setminus (N_{\Gamma^{\prime}}(v_1)\cup \left \{v_1\right \})=\emptyset$ and $V(\Gamma^{\prime})\setminus  (N_{\Gamma^{\prime}}(v_n)\cup N_{\Gamma^{\prime}}(v_1))=\emptyset$.
	\end{claim}
	\noindent\emph{Proof of Claim 6.}~Otherwise, assume that $v_i \in N_{\Gamma^{\prime}}(v_n)\setminus (N_{\Gamma^{\prime}}(v_1)\cup \left \{v_1\right \}  )$ and $v_j \in V(\Gamma^{\prime})\setminus  (N_{\Gamma^{\prime}}(v_n)\cup N_{\Gamma^{\prime}}(v_1))$. Let $\Gamma^{*}$ be a signed graph obtained from $\Gamma^{\prime}$ by deleting $v_n v_i$ and adding a positive edge $v_1 v_i$. Clearly, $\Gamma^{*}$ is also unbalanced and $t\mathcal{K}_{4}^{-}$-free. Then
	$$
	\lambda_1(\Gamma^{*})-\lambda_1(\Gamma^{\prime})\ge \mathbf{x}^{T}A(\Gamma^{*})\mathbf{x}-\mathbf{x}^{T}A(\Gamma^{\prime})\mathbf{x}=2x_i(x_1-x_n)\ge 0.
	$$
	If $\lambda_1(\Gamma^{*})=\lambda_1(\Gamma^{\prime})$, then $\mathbf{x}$ is also an eigenvector of $A(\Gamma^{*})$ corresponding to $\lambda_1(\Gamma^{*})$. Note that
	\[\lambda_1(\Gamma^{\prime})x_1=\sum_{v_k \in N_{\Gamma^{\prime}}(v_1)}\sigma^{\prime}(v_k v_1)x_k\]
	and
	\[\lambda_1(\Gamma^{*})x_1=\sum_{v_k \in N_{\Gamma^{\prime}}(v_1)}\sigma^{\prime}(v_k v_1)x_k+x_i.\]
	Then we have $x_i=0$, which contradicts Claim \ref{x>=0}. Thus, $\lambda_1(\Gamma^{*})>\lambda_1(\Gamma^{\prime})$, which contradicts the maximality of $\lambda_1(\Gamma^{\prime})$.  Let $\Gamma^{**}$ be a signed graph obtained from $\Gamma^{\prime}$ by adding a positive edge $v_1 v_j$. Clearly, $\Gamma^{**}$ is also unbalanced and $t\mathcal{K}_{4}^{-}$-free. Then
	\[\lambda_1(\Gamma^{**})-\lambda_1(\Gamma^{\prime})\ge \mathbf{x}^{T}A(\Gamma^{**})\mathbf{x}-\mathbf{x}^{T}A(\Gamma^{\prime})\mathbf{x}=2x_1 x_j> 0,\]
	which contradicts the maximality of $\lambda_1(\Gamma^{\prime})$.
	\begin{claim}
		$\Gamma[N_{\Gamma^{\prime}}(v_1)\setminus (N_{\Gamma^{\prime}}(v_n)\cup \left \{v_n\right \})]$ is a complete graph with all positive edges.
	\end{claim}
	\noindent\emph{Proof of Claim 7.}~Otherwise, assume that $v_i, v_j \in N_{\Gamma^{\prime}}(v_1)\setminus (N_{\Gamma^{\prime}}(v_n)\cup \left \{v_n\right \})$ and $v_i$ is not adjacent to $v_j$. Let $\Gamma^{*}$ be a signed graph obtained from $\Gamma^{\prime}$ by adding a positive edge $v_i v_j$. Clearly, $\Gamma^{*}$ is also unbalanced and $t\mathcal{K}_{4}^{-}$-free. Then
	\[\lambda_1(\Gamma^{*})-\lambda_1(\Gamma^{\prime})\ge \mathbf{x}^{T}A(\Gamma^{*})\mathbf{x}-\mathbf{x}^{T}A(\Gamma^{\prime})\mathbf{x}=2x_i x_j> 0,\]
	which contradicts the maximality of $\lambda_1(\Gamma^{\prime})$.
	
	\begin{claim}
		For any $v_i\in N_{\Gamma^{\prime}}(v_1)\setminus (N_{\Gamma^{\prime}}(v_n)\cup \left \{v_n\right \})$ and $v_j\in N_{\Gamma^{\prime}}(v_1)\cap N_{\Gamma^{\prime}}(v_n)$, $v_i v_j\in E(\Gamma^{\prime})$.
	\end{claim}
	\noindent\emph{Proof of Claim 8.}~Otherwise, assume that $v_i v_j\notin  E(\Gamma^{\prime})$. Let $\Gamma^{*}$ be a signed graph obtained from $\Gamma^{\prime}$ by adding a positive edge $v_i v_j$. Clearly, $\Gamma^{*}$ is also unbalanced and $t\mathcal{K}_{4}^{-}$-free. Then
	\[\lambda_1(\Gamma^{*})-\lambda_1(\Gamma^{\prime})\ge \mathbf{x}^{T}A(\Gamma^{*})\mathbf{x}-\mathbf{x}^{T}A(\Gamma^{\prime})\mathbf{x}=2x_i x_j> 0,\]
	which contradicts the maximality of $\lambda_1(\Gamma^{\prime})$.
	
	\begin{claim}
		$|N_{\Gamma^{\prime}}(v_1)\cap N_{\Gamma^{\prime}}(v_n)|=r$.	
	\end{claim}
	\noindent\emph{Proof of Claim 9.}~Note that $t\ge 2$ and $r=\frac{1+\sqrt{8t-7}}{2}\ge 2$ are positive integers. Then $ \Gamma_{r,n} $ is unbalanced and $t\mathcal{K}_{4}^{-}$-free. By Corollary \ref{Xiong-Hou-AC-2024-3}, $\lambda_1(\Gamma^{\prime})\ge \lambda_1(\Gamma_{r,n}) > \lambda_1(\Gamma_{r-1,n})$. If $\Gamma^{\prime}$ is $\mathcal{K}_{r+2}^{-}$-free, by Lemma \ref{Xiong-Hou-AC-2024-1}, $\lambda_1(\Gamma^{\prime})\le \lambda_1(\Gamma_{r-1,n})$, a contradiction. Then $\Gamma^{\prime}$ has an induced $K_{r+2}^-$ subgraph. Thus, $|N_{\Gamma^{\prime}}(v_1)\cap N_{\Gamma^{\prime}}(v_n)|\ge r$.
	
	We next prove that $|N_{\Gamma^{\prime}}(v_1)\cap N_{\Gamma^{\prime}}(v_n)|=r$ by contradiction.
	
	Assume that $|N_{\Gamma^{\prime}}(v_1)\cap N_{\Gamma^{\prime}}(v_n)|\ge r+1$. For convenience, denote $A=(N_{\Gamma^{\prime}}(v_1)\cap N_{\Gamma^{\prime}}(v_n))\setminus V(K_{r+2}^-)$ and $B= V(K_{r+2}^-)\setminus \left \{v_1, v_n\right\}$. Then for any $v\in A$ and $u\in B$, we have $vu \not \in E(\Gamma^{\prime})$. Otherwise, assume that  $vu\in E(\Gamma^{\prime})$. Clearly, $\Gamma^{\prime}[\left \{v_1,v_n,u,v\right\} ]$ is a $K_{4}^-$. Note that $K_{r+2}^-$ has a $(t-1)K_{4}^-$. Hence, $\Gamma^{\prime}$ has a $tK_{4}^-$, a contradiction. Now, $\Gamma^{\prime}\cong \Sigma_{k,n}$. By Lemma \ref{sigma_{k,n}} and the maximality of $\lambda_1(\Gamma^{\prime})$, we have $\Gamma^{\prime}\cong \Sigma_{1,n}$. Let $A=\left \{u\right \}$ and $\Gamma^{*}$ be a signed graph obtained from $\Gamma^{\prime}$ by deleting $uv_n$ and adding a positive edge $uv_j$ for any $j\in B$. Clearly, $\Gamma^{*}$ is also unbalanced and $t\mathcal{K}_{4}^{-}$-free. By the Rayleigh principle,
	\[\lambda_1(\Gamma^{*})-\lambda_1(\Gamma^{\prime})\ge \mathbf{x}^{T}A(\Gamma^{*})\mathbf{x}-\mathbf{x}^{T}A(\Gamma^{\prime})\mathbf{x} =2\left(\sum_{v_j\in B}x_j-x_n\right)x_u.\]
	Note that
	\begin{equation}\label{equation1}
		\lambda_{1}(\Gamma^{\prime})x_n=-x_1+x_u+\sum_{v_j\in B}x_j,
	\end{equation}
	\begin{equation}\label{equation2}
		\lambda_{1}(\Gamma^{\prime})x_1=-x_n+x_u+\sum_{v_j\in B}x_j+\sum_{v_k\in N_{\Gamma^{\prime}}(v_1)\setminus N_{\Gamma^{\prime}}(v_n)}x_k,
	\end{equation}
	\begin{equation}\label{equation3}
		\lambda_{1}(\Gamma^{\prime})x_u=x_1+x_n+\sum_{v_k\in N_{\Gamma^{\prime}}(v_1)\setminus N_{\Gamma^{\prime}}(v_n)}x_k.
	\end{equation}
	By Equations \eqref{equation1}, \eqref{equation2} and \eqref{equation3}, we have
	\[
	\lambda_{1}(\Gamma^{\prime})(x_1-x_u)=-x_1+x_u+\sum_{v_j\in B}x_j-2x_n= (\lambda_{1}(\Gamma^{\prime})-2 ) x_n.
	\]
	Recall that $n$ is sufficiently large. Then $\lambda_{1}(\Gamma^{\prime})>n-2>2$. Hence, $x_1\ge x_u$ with equality holding if and only if $x_n=0$.
	By Equation \eqref{equation1}, we have
	\[
	\sum_{v_j\in B}x_j-x_n=(\lambda_{1}(\Gamma^{\prime})-1)x_n+x_1-x_u \ge(\lambda_{1}(\Gamma^{\prime})-1)x_n\ge 0.
	\]
	with equality holding if and only if $x_n=0$ and $x_1=x_u$.
	Hence, $\lambda_1(\Gamma^{*})\ge\lambda_1(\Gamma^{\prime})$. If $x_n=0$ and $x_1=x_u$, by Equations \eqref{equation2} and \eqref{equation3}, we have $\sum_{v_j\in B}x_j=0$, a contradiction.
	Thus, $\lambda_1(\Gamma^{*})>\lambda_1(\Gamma^{\prime})$, which contradicts the maximality of $\lambda_1(\Gamma^{\prime})$.
	
	Therefore, $\Gamma^{\prime}=\Gamma_{r,n}$, which implies that $\Gamma \sim \Gamma_{r,n}$. This completes the proof. \qed
\end{Tproof}

Denote by $-\Gamma$ the \emph{negation} of a signed graph $\Gamma$, which is obtained by reversing the sign of each edge in $\Gamma$. Clearly, the eigenvalues of $-\Gamma$ are obtained by reversing the signs of the eigenvalues of $\Gamma$. Let $\omega_{b}(\Gamma)$ denote the \emph{balanced clique number} of $\Gamma$, which
is the order of the largest balanced clique of $\Gamma$.
\begin{lemma}\emph{(See \cite[Prosition~5]{Wang-Yan-Qian-LAA-2021})}\label{Wang-Yan-Qian-LAA-2021}
	Let $\Gamma$ be a signed graph of order $n$. Then
	\[\lambda_1(\Gamma)\le n\left(1-\frac{1}{\omega_{b}(\Gamma)}\right).\]
\end{lemma}

\begin{cor}
	Let $t\ge 2$ be a positive integer such that $r=\frac{1+\sqrt{8t-7}}{2}$ is an integer and $r\le\left\lfloor \frac{n}{2}\right\rfloor-2$. If $\Gamma$ is a $t\mathcal{K}_{4}^{-}$-free unbalanced signed graph of sufficiently large order $n$, then
	\[\rho(\Gamma)\le \rho(\Gamma_{r,n})\]
	with equality holding if and only if $\Gamma \sim \Gamma_{r,n}$.
\end{cor}
\begin{proof}
	Assume that $\Gamma$ has the maximum spectral radius among all $t\mathcal{K}_{4}^{-}$-free unbalanced signed graphs of order $n$. Note that $\Gamma_{r,n}$ is $t\mathcal{K}_{4}^{-}$-free. Then $\rho(\Gamma)\ge \rho(\Gamma_{r,n})>n-2$. If $\rho(\Gamma)\ne \lambda_1(\Gamma)$, then $\rho(\Gamma)= -\lambda_n(\Gamma)$. Since $K_{r+3}^{-}$ contains a $tK_{4}^{-}$, we obtain that $\Gamma$ is $K_{r+3}^{-}$-free. Hence, $-\Gamma$ contains no $(r+3)$-vertex balanced clique. Then $\omega_{b}(-\Gamma)\le r+2$. By Lemma \ref{Wang-Yan-Qian-LAA-2021}, we obtain that
	\[n-2<\rho(\Gamma)=-\lambda_n(\Gamma)=\lambda_1(-\Gamma)\le n\left(1-\frac{1}{\omega_{b}(-\Gamma)}\right)\le \frac{r+1}{r+2}n,\]
	which contradicts $r=\frac{1+\sqrt{8t-7}}{2}\le\left\lfloor \frac{n}{2}\right\rfloor-2$. Note that
	\[t\le \frac{\left(2\left\lfloor \frac{n}{2}\right\rfloor-5\right)^2+7}{8}\le\begin{pmatrix}
		n-2\\2
	\end{pmatrix}\text{~for~} n\ge 4.\]
	By Theorem \ref{tK4}, this completes the proof. \qed
\end{proof}


\begin{thebibliography}{99}
	\small{
		\bibitem{Brouwer-Haemers-Book-2011}
		A. Brouwer, W. Haemers, Spectra of Graphs. Springer, New York, 2011.
		
		\bibitem{Brunetti-Stanic-CAM-2022}
		M. Brunetti, Z. Stani{\'c}, Unbalanced signed graphs with extremal spectral radius or index, Comput. Appl. Math. 41 (2022) 118.
		
		\bibitem{Cam-Sei-Tsa-Algebra-1994}
		P. J. Cameron, J. J. Seidel, S. V. Tsaranov, Signed graphs, root lattices, and Coxeter groups, J. Algebra 164 (1) (1994) 173--209.
		
		\bibitem{Cartwright-Harary-Psy-1956}
		D. Cartwright, F. Harary, Structural balance: a generalization of Heider's theory, Psychol. Rev. 63 (1956) 277--293.
		
		\bibitem{Chaiken-SIAM-1982}
		S. Chaiken, A combinatorial proof of the all minors matrix tree theorem, SIAM J. Algebraic Discrete Methods 3 (1982) 319--329.
		
		\bibitem{Chen-Yuan-AMC-2024}
		F. Chen, X. Yuan, Tur{\'a}n problem for $\mathcal{K}_4^-$-free signed graphs, Appl. Math. Comput. 477 (2024) 128814.
		
		\bibitem{Chen-Lei-Li-EJC-2025}
		G. Chen, X. Lei, S. Li, The exact Tur{\'a}n number of disjoint graphs--a generalization of Simonovits' theorem, and beyond, European J. Combin. 130 (2025) 104226.
		
		\bibitem{Har-MichM-1953}
		F. Harary, On the notion of balance of a signed graph, Michigan Math. J. 2 (2) (1953) 143--146.
		
		\bibitem{Hei-JP-1946}
		F. Heider, Attitudes and cognitive organization, J. Psychol. 21 (1) (1946) 107--112.
		
		\bibitem{Hou-Tang-Wang-AMC-2019}
		Y. Hou, Z. Tang, D. Wang, On signed graphs with just two distinct Laplacian eigenvalues, Appl. Math. Comput. 351 (2019) 1--7.
		
		\bibitem{Li-Qin-Ar-2025}
		D. Li, M. Qin, The index of $t\mathcal{C}_{3}^-$-free signed graphs, arXiv:2512.07579 (2025).
		
		\bibitem{Nikiforov-LAA-2007}
		V. Nikiforov, Bounds on graph eigenvalues II. Linear Algebra Appl. 427 (2007) 183--189.
		
		\bibitem{Nikiforov-LAA-2010}
		V. Nikiforov, The spectral radius of graphs without paths and cycles of specified length, Linear Algebra Appl. 432 (2010) 2243--2256.
		
		\bibitem{Nikiforov-LAA-2017}
		V. Nikiforov, The spectral radius of graphs with no $K_{2,t}$ minor, Linear Algebra Appl. 531 (2017) 510--515.
		
		\bibitem{Ni-Wang-Kang-EJC-2023}
		Z. Ni, J. Wang, L. Kang, Spectral extremal graphs for disjoint cliques, Electron. J. Combin. 30 (1) (2023) P1.20.
		
		\bibitem{Sun-Liu-Lan-LAA-2022}
		G. Sun, F. Liu, K. Lan, A note on eigenvalues of signed graphs, Linear Algebra Appl. 652 (2022) 125--131.
		
		\bibitem{Wang-Hou-Li-LAA-2024}
		D. Wang, Y. Hou, D. Li, Extremal results for $C_3^-$-free signed graphs, Linear Algebra Appl. 681 (2024) 47--65.
		
		\bibitem{Wang-Hou-Huang-DAM-2025}
		J. Wang, Y. Hou, X. Huang, Tur{\'a}n problem of signed graph for negative odd cycle, Discrete Appl. Math. 362 (2025) 157--166.
		
		\bibitem{Wang-Yan-Qian-LAA-2021}
		W. Wang, Z. Yan, J. Qian, Eigenvalues and chromatic number of a signed graph, Linear Algebra Appl. 619 (2021) 137--145.
		
		\bibitem{Wang-LAA-2024}
		Y. Wang, Spectral Tur{\'a}n problem for $\mathcal{K}^-_5$-free signed graphs, Linear Algebra Appl. 691 (2024) 96--108.
		
		\bibitem{Wang-Lin-DAM-2024}
		Y. Wang, H. Lin, The largest eigenvalue of $\mathcal{C}^-_k$-free signed graphs, Discrete Appl. Math. 372 (2025) 164--172.
		
		\bibitem{Xiong-Hou-AC-2024}
		Z. Xiong, Y. Hou, Extremal results for $\mathcal{K}^-_{r+1}$-free unbalanced signed graphs, Ars Combin. 161 (2024) 61--73.
		
		\bibitem{Yuan-Wang-Zhai-EJLA-2012}
		W. Yuan, B. Wang, M. Zhai, On the spectral radii of graphs without given cycles, Electron. J. Linear Algebra 23 (2012) 599--606.
		
		\bibitem{Zaslavsky-DAM-1982}
		T. Zaslavsky, Signed graphs, Discrete Appl. Math. 4 (1) (1982) 47--74.
		
		\bibitem{Zaslavsky-book-2010}
		T. Zaslavsky, Matrices in the theory of signed simple graphs, in Advances in Discrete Mathematics and Applications: Mysore, 2008, 207--229, Ramanujan Math. Soc. Lect. Notes Ser., 13, Ramanujan Math. Soc., Mysore, 2010.
		
		\bibitem{Zaslavsky-EJC-2018}
		T. Zaslavsky, A mathematical bibliography of signed and gain graphs and allied areas, Electron. J. Combin. 5 (2018), Dynamic Surveys 8, 524 pages.
		
		
	}
\end{thebibliography}
\end{document}